\documentclass[11pt,a4paper]{article}
\usepackage{amssymb,amsmath}
\usepackage{tikz}
\usetikzlibrary{arrows,patterns}

\usepackage{amssymb}

\newtheorem{theorem}{Theorem}[section]
\newtheorem{corollary}[theorem]{Corollary}

\newtheorem{lemma}[theorem]{Lemma}
\newtheorem{proposition}[theorem]{Proposition}
\newtheorem{remark}[theorem]{Remark}

\newtheorem{conjecture}[theorem]{Conjecture}
\newtheorem{example}[theorem]{Example}

\newenvironment{proof}{\begin{trivlist}\item[]{\it Proof.}}
{\hfill$\square$\end{trivlist}}

\author{M. Domokos  ${}^a$ \thanks{Partially supported by OTKA K101515.}\;  \ and M. Zubor ${}^b$ 
\\ 
\\ 
{\small ${}^a$ R\'enyi Institute of Mathematics, Hungarian Academy of 
Sciences,} 
\\ {\small 1053 Budapest, Re\'altanoda utca 13-15., Hungary } 
\\{\small E-mail: domokos.matyas@renyi.mta.hu } 
\\ 
\\
{\small ${}^b$ Department of Algebra, Budapest University of Technology and Economics }
\\ {\small 1521 Budapest, P. O. Box 91, Hungary}
 \\ {\small E-mail: zuborm@math.bme.hu  } 
}

\title{Commutative subalgebras of the Grassmann algebra}
\date{}
\begin{document}

\maketitle

\begin{abstract}
The maximal dimension of a  commutative subalgebra of the Grassmann algebra is determined.  It is shown that for any commutative subalgebra there exists 
an equidimensional commutative subalgebra   spanned by monomials. 
It follows that the maximal dimension of a commutative subalgebra can be expressed in terms of the maximal size of an intersecting system of subsets of odd size in a finite set. 
\end{abstract}

\noindent {\it 2010 MSC:} 15A75 (Primary); 05D05 (Secondary);  16W55 (Secondary); 13A02 (Secondary)

\noindent{\it Keywords:} Grassmann algebra, commutative subalgebra, intersecting set systems, commuting projections

\section{Introduction}\label{sec:intro} 

The {\it Grassmann algebra} (called also  {\it exterior algebra}) $E := E^{(n)}$ of an $n$-dimensional vector space  $\mathrm{Span}_{\mathbb{F}}\{v_1,\dots,v_n\}$ over a field $\mathbb{F}$ (assumed throughout to have characteristic different from $2$) is the associative $\mathbb{F}$-algebra given in terms of generators and relations as  
\[E^{(n)}=\mathbb{F}\langle v_1,...,v_n\mid \quad v_iv_j=-v_jv_i  \quad (1\le i,j\le n)\rangle.\] 
In the present paper we shall investigate the algebra structure of this fundamental object. The algebra $E$ is not commutative, but it is not far from being commutative: it has a large center, and satisfies the polynomial identity $[[x,y],z]=0$ (Lie nilpotency of index $2$). 
We are  interested  in the commutative subalgebras of $E$. Our main result Theorem~\ref{thm:maxcommsubalg} gives in particular the maximal dimension of a 
commutative subalgebra of $E$, and gives some partial results on their structure. It turns out that in case when $n$ is even,  all maximal (with respect to inclusion) subalgebras have the same dimension. When $n$ is odd,  there are maximal commutative subalgebras with different dimensions. 

Along the way we show in Theorem~\ref{thm:monomialsubalg} that to any commutative subalgebra of $E$ one can associate via a simple linear algebra process a sequence of equidimensional commutative subalgebras, the last one spanned by monomials (products of generators). This result has some independent interest, and establishes a tight connection between our question and the Erd\H os-Ko-Rado Theorem on intersecting set systems. 

The study of  commutative subalgebras in non-commutative algebras has a considerable literature. We mention only the theorem of  Schur \cite{schur} determining the maximal dimension of a commutative subalgebra of $\mathbb{F}^{n\times n}$, see \cite{jacobson} and \cite{gustafson} for alternative proofs. 
Our interest in maximal commutative subalgebras of the Grassmann algebra was inspired by \cite{szigeti-etal}, where the existence of large commutative subalgebras of $E$ is used as an obstruction for the existence  of an embedding of $E$  into the full matrix algebra $\mathbb{F}^{m\times m}$ with small $m$ (see Theorem 3.9 of \cite{szigeti-etal}).


\section{Square zero subspaces}\label{sec:squarezero}

For a subset $J\subseteq [n]:=\{1,\dots,n\}$ set $v_J:=v_{i_1}\cdots v_{i_k}$, where $J=\{i_1,\dots,i_k\}$ and $i_1<\dots<i_k$. 
Clearly, $\{v_J\mid J\subseteq [n]\}$ is an $\mathbb{F}$-vector space basis of $E^{(n)}$. We shall refer to the elements $v_J\in E$ as {\it monomials}. 
The Grassmann algebra is graded: 
\[E^{(n)}=\bigoplus_{k=0}^\infty E^{(n)}_k\mbox{ where }E^{(n)}_k=\mathrm{Span}_{\mathbb{F}}\{v_J\mid I\subseteq [n], \quad|J|=k\}\] 
(of course, for $k>n$ we have $E^{(n)}_k=\{0\}$). Sometimes we pay attention to the $\mathbb{Z}/2\mathbb{Z}$-grading induced by the above $\mathbb{Z}$-grading: 
\[E^{(n)}=E^{(n)}_{\overline{0}}\oplus E^{(n)}_{\overline{1}}\mbox{ where }E^{(n)}_{\overline{0}}:=\bigoplus_{k\textrm{ is  even}} E^{(n)}_k, \quad  
E^{(n)}_{\overline{1}}:=\bigoplus_{k\textrm{ is  odd}} E^{(n)}_k\]

The defining relations of $E:=E^{(n)}$ imply  the multiplication rules $v_Jv_K=(-1)^{|J|\cdot|K|}v_Kv_J$ and when $J\cap K\neq\emptyset$, we have $v_Jv_K=0$. 
It follows that $E_{\overline{0}}$ is contained in the center of $E$, and  the elements of $E_{\overline{1}}$ anticommute: $ab = -ba$ for any pair
$a, b \in E_{\overline{1}}$. In particular, $a, b \in E_{\overline{1}}$ commute if and only if $ab = 0$. So the commutative subalgebras of $E$ have a natural connection with 
square zero subspaces. For subspaces $C,D\subseteq E$ we write $CD$ for the subspace $\mathrm{Span}_{\mathbb{F}}\{cd\mid c\in C, d\in D\}$, and we call a subspace  
$D\subseteq E$ a {\it square zero subspace} if  $D^2=0$, that is, if $cd=0$ for all $c,d\in D$.  
A commutative subalgebra $A$  of $E$ is called  {\it maximal} if there is no commutative subalgebra of $E$ properly containing $A$. 
Similarly, a square zero subspace of $E_{\overline{1}}$ is called {\it maximal} if it is not properly contained in a square zero subspace of $E_{\overline{1}}$. 

\begin{proposition}\label{prop:strukt1} 
\begin{enumerate}
\renewcommand{\labelenumi}{(\roman{enumi})}
\item If $D\subseteq E_{\overline{1}}$ is a square zero subspace,  
then $K := E_{\overline{0}} D \subseteq E_{\overline{1}}$ is also a square zero subspace and $E_{\overline{0}} \oplus K$ is a commutative subalgebra of $E$.
\item The map $D\mapsto E_{\overline{0}}\oplus D$ gives a bijection between the maximal square zero subspaces in $E_{\overline{1}}$ and maximal commutative subalgebras of $E$. 
\end{enumerate}
\end{proposition}

\begin{proof}
The statement (i)  follows from the centrality of $E_{\overline{0}}$. 
In order to show (ii), suppose that $A\supseteq E_{\overline{0}}$ is a commutative subalgebra in $E$. Then  as a vector space, $A = E_{\overline{0}} \oplus D$, where $D := A \cap E_{\overline{1}}$. Moreover, $D$ is an $E_{\overline{0}}$-submodule in $E_{\overline{1}}$.  
Since elements of $E_{\overline{1}}$ anticommute, this forces that $D^2 = 0$. 
Taking into account (i) we get that $D\mapsto E_{\overline{0}}\oplus D$ gives a bijection beween square zero $E_{\overline{0}}$-submodules of $E_{\overline{1}}$ and commutative subalgebras of $E$ that contain $E_{\overline{0}}$. This bijection restricts to the bijection claimed in (ii), since a maximal commutative subalgebra of $E$ necessarily contains the center $E_{\overline{0}}$, and a maximal square zero subspace of $E_{\overline{1}}$ is necessarily an $E_{\overline{0}}$-submodule.  
\end{proof}

\begin{remark}\label{remark:schur} 
{\rm It is interesting to compare the above "structure theorem" with the case of the matrix algebra $\mathbb{F}^{n\times n}$, where by a theorem of Schur \cite{schur}, a commutative subalgebra of maximal possible dimension is also of the form $\mathrm{Center}(\mathbb{F}^{n\times n})\oplus D=\mathbb{F}I_n  \oplus D$, where $D$ is a subspace with $D^2=0$ and $I_n$ stands for the identity matrix. However, unlike for  $E$,  in  $\mathbb{F}^{n\times n}$ not all maximal subalgebras are of this form. }
\end{remark}

Maximal square zero subspaces in $E_{\overline{1}}$ can be characterized in terms of certain bilinear maps on $E_{\overline{1}}$ defined as the composition of the multiplication map  $E_{\overline{1}}\times E_{\overline{1}}\to E_{\overline{0}}$ and a projection from $E_{\overline{0}}$ to one of its homogeneous components. 
Recall that any $x\in E^{(n)}$ can be uniquely written as 
\begin{equation}\label{eq:x_J} x=\sum_{J\subseteq [n]}x_Jv_J.
\end{equation} 
Define a  bilinear map $\phi$ as follows: 
\begin{itemize}
\item if $n$ is even then $\Phi:E_{\overline{1}}\times E_{\overline{1}}\rightarrow E_n$, $\Phi(a,b)=(ab)_{[n]}v_{[n]}$;  
\item if $n$ is odd then $\Phi:E_{\overline{1}}\times E_{\overline{1}}\rightarrow E_{n-1}$, $\Phi(a,b)=\sum_{J\in \binom{[n]}{n-1}}(ab)_Jv_J$. 
\end{itemize}  
Here $\binom{[n]}{k}$ stands for the set of $k$-element subsets of $[n]$. 
Since the multiplication map on $E_{\overline{1}}$ is skew-symmetric, the bilinear map $\Phi$ is also skew-symmetric. 
Moreover,  $E^{(n)}_{n}$ can be identified with $\mathbb{F}$, so $\Phi$ is a skew-symmetric bilinear form when $n$ is even. It is a non-degenerate form, since if $x_J\neq 0$ for some $J\subseteq[n]$ and $x\in E_{\overline{1}}$, then $v_{[n]\setminus J}\in E_{\overline{1}}$ and 
$\Phi(x,v_{[n]\setminus J})=x_J\neq 0$. So for $n$ even $(E_{\overline{1}},\Phi)$ is a symplectic vector space. 

Given a subspace  $D\subseteq E_{\overline{1}}$ we write 
\[D^\perp:=\{x\in E_{\overline{1}}:\Phi(x,w)=0\text{ }\forall w\in D  \}.\]

\begin{proposition}\label{prop:strukt2}
 A subspace $D\subseteq E_{\overline{1}}$ is a maximal square zero subspace in $E_{\overline{1}}$ if and only if $D$ is an $E_{\overline{0}}$-submodule  and 
 $D=D^\perp$.  
\end{proposition}

\begin{proof} 
Let $D$ be an $E_{\overline{0}}$-submodule  of $E_{\overline{1}}$ with  $D=D^\perp$. Suppose that there exist $x,y\in D$ with $xy\ne 0$. Then 
there exists a (homogeneous) $z\in E_{\overline{0}}$ such that $0\ne xyz=\Phi(x,yz)$, hence $yz\notin D^\perp$. This is a contradiction, since  $yz\in DE_{\overline{0}}=D=D^\perp$. Thus $D^2=0$. 

To show the reverse implication suppose that $D\subseteq E_{\overline{1}}$ is a maximal square zero subspace in $E_{\overline{1}}$. By Proposition~\ref{prop:strukt1} (i)   (and since $1\in E_{\overline{0}}$) we have $D= DE_{\overline{0}}$. 
Clearly $D^2=0$ implies $\Phi(x,y)=0$ for all $x,y\in D$, so  $D\subseteq D^\perp$. 
Moreover, if $x\in E_{\overline{1}}\setminus D$ then by maximality of $D$  there exists $y\in D$ such that $0\ne xy\in E_{\overline{0}}$. 
Thus there  exists a $z\in E_{\overline{0}}$ such that $0\ne xyz=\Phi(x,yz)$. Here $yz\in DE_{\overline{0}}=D$, showing $x\notin D^\perp$. So  we proved $D^\perp\subseteq D$. Taking into account the reverse inclusion we get $D=D^\perp$, and the proof is finished. 
\end{proof}

\begin{corollary}\label{cor:even}
If $n\ge 2$ is even, then any  maximal commutative subalgebra of   $E^{(n)}$ has dimension $3\cdot 2^{n-2}$.
\end{corollary}
\begin{proof} Let $A$ be a maximal commutative subalgebra of $E^{(n)}$. Then by Proposition~\ref{prop:strukt1} (ii) $A=E_{\overline{0}}\oplus D$, where  
$D\subseteq E_{\overline{1}}$ is a maximal square zero subspace. By Proposition~\ref{prop:strukt2} we get $D=D^\perp$ is a Lagrangian subspace in the symplectic vector space $(E_{\overline{1}},\Phi)$,  hence $\dim(D)=\dim(E_{\overline{1}})/2=2^{n-2}$. 
So $\dim(A)=\dim(E_{\overline{0}})+2^{n-2}=3\cdot 2^{n-2}$.  
\end{proof}


\section{Commuting projections}\label{sec:projections} 

Let $W$ be a vector space over a field  $\mathbb{F}$, and $\pi_1,\dots,\pi_r\in\mathrm{End}_{\mathbb{F}}(W)$ pairwise commuting projections, so $\pi_i^2=\pi_i$ and $\pi_i\pi_j=\pi_j\pi_i$ 
for all $i,j\in [r]$. Recall the  corresponding  direct sum decompositions $W=\ker(\pi_j)\oplus \mathrm{im}(\pi_j)$. 
Given a subset $J \subseteq [r]$ we set 
\[W_J:=\bigcap_{j\in J} \ker(\pi_j)\cap \bigcap_{j\notin J} \mathrm{im}(\pi_j).\]
Let $\mathrm{Gras}(W)$ stand for the set of subspaces of $W$, and for $j=1,\dots,r$ define a map 
\begin{equation}\label{eq:gammadef}\gamma_j:\mathrm{Gras}(W)\to \mathrm{Gras}(W), \qquad 
D\mapsto \ker(\pi_j\vert_D)\oplus \mathrm{im}(\pi_j\vert_D)\end{equation} 
where $\pi_j\vert_D:D\to W$ stands for the restriction of $\pi_j$ to the subspace  $D\subseteq W$. 
Note that for $A,D\in\mathrm{Gras}(E^{(n)})$ we have $\gamma_j(A)+\gamma_j(D)\subseteq \gamma_j(A+D)$ (this inclusion is proper in general). It is also obvious that 
\begin{equation}\label{eq:gamma-subspace}
A\subseteq D  \quad \mbox{ implies }\quad \gamma_j(A)\subseteq \gamma_j(D).\end{equation}  

\begin{lemma}\label{lemma:decomposition}
Take $D\in\mathrm{Gras}(W)$ and denote $A:=\gamma_1\dots\gamma_r(D)$. Then we have the equalities 
\begin{itemize}
\item[(i)] $\dim(A)=\dim(D)$; 
\item[(ii)] $A=\bigoplus_{J\subseteq[r]}(A\cap W_J).$ 
\end{itemize}
\end{lemma} 

\begin{proof} The case $r=1$ is a basic fact of linear algebra. 
Statement (i) follows by a repeated application of the special case $r=1$ of (i). To prove statement (ii) we apply induction on $r$.  Suppose $r>1$. 
Set $W':=\ker(\pi_r)$, $W'':=\mathrm{im}(\pi_r)$. Since  $\pi_j\pi_r=\pi_r\pi_j$ for $j=1,\dots,r-1$, it follows that $\pi_j(W')\subseteq W'$, $\pi_j(W'')\subseteq W''$, 
and if $C=C'+C''$ with $C'\subseteq W'$, $C''\subseteq W''$, then 
\begin{equation}\label{eq:W'W''}\gamma_j(C')\subseteq W', \quad  \gamma_j(C'')\subseteq W'', \mbox{ and }
\gamma_j(C)=\gamma_j(C')\oplus \gamma_j(C''). \end{equation} 
 By definition of $\gamma_r$ we have $\gamma_r(D)=(\gamma_r(D)\cap W')\oplus (\gamma_r(D)\cap W'')$, hence by a repeated application of \eqref{eq:W'W''} we get  
\begin{equation}\label{eq:A}A=(\gamma_1\dots\gamma_{r-1})(\gamma_r(D))=(A\cap W')\oplus (A\cap W'')\end{equation}
\begin{equation}\label{eq:AcapW'} 
 A\cap W'=\gamma_1\dots \gamma_{r-1} (\gamma_r(D)\cap W') \end{equation} 
\begin{equation}\label{eq:AcapW''}A\cap W''=\gamma_1\dots \gamma_{r-1}(\gamma_r(D)\cap W''). \end{equation}   
For a subset $J\subseteq [r-1]$ set 
\[W'_J:=\bigcap_{j\in J} \ker(\pi_j\vert_{W'})\cap \bigcap_{j\in [r-1]\setminus J} \mathrm{im}(\pi_j\vert_{W'})\]
\[W''_J:=\bigcap_{j\in J} \ker(\pi_j\vert_{W''})\cap \bigcap_{j\in [r-1]\setminus J} \mathrm{im}(\pi_j\vert_{W''}).\] 
Applying the induction hypothesis to \eqref{eq:AcapW'} and \eqref{eq:AcapW''} we get 
\begin{equation}\label{eq:induction} A\cap W'=\bigoplus_{J\subseteq [r-1]} A\cap W'_J \quad \mbox{ and  }\quad 
A\cap W''=\bigoplus_{J\subseteq [r-1]} A\cap W''_J.\end{equation}  
Observe that 
\begin{equation}\label{eq:W'_J} W'_J=W_{J\cup \{r\}}\quad\mbox{ and }\quad W''_J=W_J. \end{equation} 
Now  \eqref{eq:A}, \eqref{eq:induction} and 
\eqref{eq:W'_J} yield the statement (ii). 
\end{proof} 

\begin{remark}{\rm Note that though the $\pi_1,\dots,\pi_r$ commute, the maps $\gamma_1,\dots,\gamma_r$ do not commmute, i.e. 
$\gamma_i\gamma_j(D)$ may be different from $\gamma_j\gamma_i(D)$ (see Example \ref{ex:1} and \ref{ex:4}).}
\end{remark}


\section{Projections on the Grassmann algebra}\label{sec:projectionsongrassmann}

We shall apply Lemma~\ref{lemma:decomposition} for the case when $W=E^{(n)}$ is the Grassmann algebra. 
For $i=1,\dots n$ define the linear map $\pi_i:E^{(n)}\to E^{(n)}$ by $\pi_i(x):=\sum_{i\notin J\subseteq [n]}x_Jv_J$ (see \eqref{eq:x_J} for the notation). 
Then $\pi_i^2=\pi_i$ and $\pi_i\pi_j=\pi_j\pi_i$, so the considerations of Section~\ref{sec:projections} apply for these projections. 
Keeping the notation of Section~\ref{sec:projections} define $\gamma_i:\mathrm{Gras}(E^{(n)})\to \mathrm{Gras}(E^{(n)})$ as in 
\eqref{eq:gammadef}. Observe that $E^{(n)}_J=W_J$ is the $1$-dimensional subspace spanned by $v_J$. 
An extra feature now is that the $\pi_i$ are algebra homomorphisms, moreover,  
\begin{equation}\label{eq:squarezero} 
\ker(\pi_i)^2=\{0\}\end{equation} 
since $\ker(\pi_i)=v_iE^{(n)}=E^{(n)}v_i$ and $v_i^2=0$.  

\begin{proposition}\label{prop:gammaproperties} 
 Let $D,A\in \mathrm{Gras}(E^{(n)})$ be subspaces. The following hold for $\gamma_i$: 
\begin{itemize}
\item[(i)]  We have $\gamma_i(A)\gamma_i(D)\subseteq \gamma_i(AD)$. 
\item[(ii)] If $D$ is a subalgebra of $E^{(n)}$, then $\gamma_i(D)$ is also a subalgebra. 
\item[(iii)] If $D^2=\{0\}$, then $\gamma_i(D)^2=\{0\}$. 
\item[(iv)] If $D$ is a right/left ideal in $E^{(n)}$ then $\gamma_i(D)$ is also a right/left ideal in $E^{(n)}$.
\item[(v)] If $D$ is a commutative subalgebra of $E^{(n)}$, then $\gamma_i(D)$ is  a commutative subalgebra of $E^{(n)}$. 
\end{itemize}  
\end{proposition}

\begin{proof} To simplify notation set $\gamma:=\gamma_i$ and $\pi:=\pi_i$. 

(i) Note that $\ker(\pi)^2=\{0\}$ implies that for any $x\in E^{(n)}$ and $y\in\ker(\pi)$ we have 
\begin{equation}\label{eq:pi(x)y} 
xy=\pi(x)y\mbox{ and }yx=y\pi(x).\end{equation}
It follows that $\pi(A)\ker(\pi\vert_D)=A\ker(\pi\vert_D)\subseteq \ker(\pi\vert_{AD})$ and similarly 
$\ker(\pi\vert_A)\pi(D)\subseteq \ker(\pi\vert_{AD})$. Taking into account that $\pi(A)\pi(D)=\pi(AD)$ we conclude
\[\gamma(A)\gamma(D)=(\ker(\pi\vert_A)+\mathrm{im}(\pi\vert_A))(\ker(\pi\vert_D)+\mathrm{im}(\pi\vert_D))
\subseteq \ker(\pi)^2+\ker(\pi\vert_{AD})+\pi(A)\pi(D)=\gamma(AD).\] 

Statements (ii), (iii), (iv) are immediate corollaries of (i). 

(v) A general pair of elements in $\gamma(D)$ can be written as $\pi(a)+a'$ and $\pi(b)+b'$ where $a,b\in D$ and $a',b'\in\ker(\pi\vert_D)$. 
From \eqref{eq:pi(x)y}, the commutativity of $D$, and that $\pi$ is an algebra homomorphism it follows that 
\[(\pi(a)+a')(\pi(b)+b')=\pi(ab)+a'b+ab'+a'b'=\pi(ba)+ba'+b'a+b'a'=(\pi(b)+b')(\pi(a)+a').\]
\end{proof} 

\begin{remark}{\rm Note that in the situation of Proposition~\ref{prop:gammaproperties} (ii) the algebra $\gamma_i(D)$ is not necessarily isomorphic to the algebra $D$ (see Example \ref{ex:2} and \ref{ex:6}).}
\end{remark} 

Combining Lemma~\ref{lemma:decomposition}, Proposition~\ref{prop:gammaproperties} and the fact that $E^{(n)}_J=\mathbb{F}v_J$ 
we obtain the following: 

\begin{theorem}\label{thm:monomialsubalg} 
Let $D\subseteq E^{(n)}$ be a subalgebra (not necessarily unitary), and set $A:=\gamma_1\dots\gamma_n(D)$. 
Then $A$ is a subalgebra of $E^{(n)}$ spanned as an $\mathbb{F}$-vector space by elements of the form $v_J$, $J\subseteq [n]$, and 
$\dim(A)=\dim(D)$. Moreover, if $D$ is commutative then $A$ is commutative, and if $D^2=\{0\}$ then $A^2=\{0\}$.  
\end{theorem} 

\begin{remark}\label{remark:other-gamma} 
{\rm The role of the generators $v_1,\dots,v_n$ is symmetric, so the conclusion of Theorem~\ref{thm:monomialsubalg} holds for $A=\gamma_{\sigma(1)}\dots\gamma_{\sigma(n)}(D)$ where $\sigma$ is an arbitrary permutation of $1,\dots,n$. However, different permutations $\sigma$ yield in general different subspaces $\gamma_{\sigma(1)}\dots\gamma_{\sigma(n)}(D)$, see Example~\ref{ex:4}.} 
\end{remark} 

The projections $\pi_i$ preserve the degree, hence the maps $\gamma_i$ are also compatible with the grading on $E$: 

\begin{proposition}\label{prop:gamma-grading} 
\begin{itemize}
\item[(i)] If $D\subseteq \bigoplus_{k\in I}E_k$ for some $I\subseteq [n]$, then $\gamma_i(D)\subseteq \bigoplus_{k\in I}E_k$. 
\item[(ii)] If $D\subseteq \bigoplus_{k\in I}E_k$ and $A\subseteq \bigoplus_{k\in J}E_k$ where  $I,J\subseteq[n]$ are disjoint subsets then $\gamma_i(A\oplus D)=\gamma_i(A)\oplus\gamma_i(D)$. 
\item[(iii)] If $D=\bigoplus_{k=0}^n(D\cap E_k)$ is spanned by its homogeneous components, then we have 
$\gamma_i(D)=\bigoplus_{k=0}^n\gamma_i(D\cap E_k)$ (and $\gamma_i(D\cap E_k)\subseteq E_k$ for all $k$). 
\end{itemize}
\end{proposition}

Write $b^{\min}$ for the minimal degree non-zero homogeneous component of a non-zero $b\in E$. 
To any subspace $A$ of $E$ one can canonically associate a subspace $A^{\min}$ spanned by homogeneous elements as follows: 
\[A^{\min}:=\mathrm{Span}_{\mathbb{F}}\{a\in E\mid a=b^{\min}  \mbox{ for some non-zero }b\in A\}.\] 
For later use we mention the following straightforward statement: 

\begin{proposition}\label{prop:a^min} 
\begin{itemize} 
\item[(i)] $\dim(A^{\min})=\dim(A)$; 
\item[(ii)] If $A$ is a subalgebra of $E$ then $A^{\min}$ is a subalgebra of $E$. 
Moreover, if $A$ is commutative then $A^{\min}$ is commutative, and if $A$ is a square zero subspace then $A^{\min}$ is a square zero subspace. 
\end{itemize}
\end{proposition} 

Note also that for any graded subalgebra $A$ of $E$, the subalgebra $B:=\gamma_1\dots\gamma_n(A)$ spanned by monomials has the same Hilbert series as $A$: we have 
$\dim(A\cap E_k)=\dim(B\cap E_k)$  for $k=0,1,\dots,n$ by Proposition~\ref{prop:gamma-grading}. 

\begin{remark} \label{remark:notallgraded} {\rm Not all subalgebras of $E^{(n)}$ are isomorphic to a graded subalgebra of $E^{(n)}$, and not all graded subalgebras of $E$ are isomorphic to  a subalgebra generated by monomials, see Example~\ref{example:nongraded} and \ref{example:nonmonomial}.  }
\end{remark} 


\section{Initial monomials}\label{sec:groebner} 

The end result of the construction in Theorem~\ref{thm:monomialsubalg} can be interpreted in terms of initial monomials. 
Given a total order $>$ on the set $\{v_I\mid I\subseteq [n]\}$ of monomials, for a non-zero $x\in E$ as in \eqref{eq:x_J} set  $\mathrm{inm}_>(x):=v_I$, where $v_I$ is the largest with respect to $>$ among the monomials $v_J$ with $x_J\neq 0$. The {\it initial term} of $x$ is $\mathrm{int}_>(x):=x_Iv_I$ where $v_I=\mathrm{inm}_>(x)$. 
We define $\mathrm{int}_>(0):=0$. 
For a subspace $A$ in $E$ set $\mathrm{in}_>(A):=\mathrm{Span}_{\mathbb{F}}\{\mathrm{inm}_>(x)\mid x\in A\}$. 
It is straightforward that  $\dim(A)=\dim(\mathrm{in}_>(A))$ (see for example Proposition 1.1 in \cite{aramova-herzog-hibi}). 
Suppose that the monomial order $>$ satisfies the following: 
\begin{equation}\label{eq:order-condition}
\mbox{If } \quad \mathrm{inm}_>(x)\mathrm{inm}_>(y)\neq 0 \quad \mbox{ then }\quad \mathrm{int}_>(xy)=\mathrm{int}_>(x)\mathrm{int}_>(y).
\end{equation} 
Condition \eqref{eq:order-condition} holds for example for the monomial orders considered in \cite{aramova-herzog-hibi}, where the basics of Gr\"obner basis theory in the Grasmann algebra are worked out. As we shall see below, there are other interesting monomial orders that satisfy \eqref{eq:order-condition}, but do not fit into the framework of 
\cite{aramova-herzog-hibi}. 
Note that \eqref{eq:order-condition} implies that for  $x,y\in E$ either we have 
$\mathrm{int}_>(x)\mathrm{int}_>(y)=0=\mathrm{int}_>(y)\mathrm{int}_>(x)$ 
or otherwise  $\mathrm{int}_>(x)\mathrm{int}_>(y)=\mathrm{int}_>(xy)$ and 
$\mathrm{int}_>(y)\mathrm{int}_>(x)=\mathrm{int}_>(yx)$. 
This yields the following: 

\begin{proposition}\label{prop:insubalg} If  $A$ is a subalgebra in $E$ and $>$ satisfies \eqref{eq:order-condition}, then $\mathrm{in}_>(A)$ is an equidimensional subalgebra as well.  In addition, if  $A$ is commutative, then $\mathrm{in}_>(A)$ is commutative, and if 
$A$ is a square zero subspace, then $\mathrm{in}_>(A)$ is a square zero subspace.  
\end{proposition} 

The connection between Theorem~\ref{thm:monomialsubalg} and Proposition~\ref{prop:insubalg} is explained in the following statement: 

\begin{proposition}\label{prop:gamma-initial} Let $D\subseteq E^{(n)}$ be a subspace, and  set $A:=\gamma_1\dots\gamma_n(D)$. 
Then we have $A=\mathrm{in}_>(D)$, where $>$ is the monomial order defined as follows: $v_I>v_J$ for some 
$I=\{i_1>\cdots >i_r\}$ and $J:=\{j_1>\cdots > j_s\}$ if $s>r$ and $i_1=j_1,\dots,i_r=j_r$, or otherwise for some $k\le\min\{r,s\}$ we have $i_1=j_1,\dots,i_{k-1}=j_{k-1}$, 
and $i_k<j_k$. 
\end{proposition}

\begin{proof}  Take any non-zero $x\in D$, set $x^{(0)}:=x$, and for $k=1,\dots,n$ define $x^{(k)}$ recursively as follows: if $v_{n-k+1}$ is involved in each non-zero term  
of  $x^{(k-1)}$, then $x^{(k)}:=x^{(k-1)}$, otherwise $x^{(k)}$ is obtained by removing all terms of $x^{(k-1)}$ that involve $v_{n-k+1}$. It is clear from the definition of 
$\gamma_j$ that $\gamma_{n-k+1}(\mathbb{F}x^{(k-1)})=\mathbb{F}x^{(k)}$, hence by \eqref{eq:gamma-subspace} we conclude $x^{(n)}\in A$. On the other hand, 
by induction on $k$ we show that $\mathrm{inm}_>(x^{(k)})=\mathrm{inm}_>(x)$ for all $k$. This is trivial for $k=0$; assume that $k>0$, and $\mathrm{inm}_>(x^{(k-1)})=\mathrm{inm}_>(x)$. For each $j\in\{n-k+2,\dots,n\}$ we have that either $v_j$ is involved in each term of $x^{(k-1)}$, or $v_j$ is not involved in any non-zero term of $x^{(k-1)}$.  It follows by definition of $>$ that if $v_{n-k+1}$ is involved in the initial term of $x^{(k-1)}$, then it is involved in all its terms, hence $x^{(k-1)}=x^{(k)}$, or $v_{n-k+1}$ is not involved 
in the initial term of $x^{(k-1)}$, hence the initial term of $x^{(k-1)}$  survives in $x^{(k)}$. In either case we get $\mathrm{inm}_>(x^{(k)})=\mathrm{inm}_>(x^{(k-1)})$.   
Note finally that  by Lemma~\ref{lemma:decomposition} $x^{(n)}$ is a non-zero scalar multiple of a monomial, hence 
$\mathrm{inm}_>(x)\in \mathbb{F} x^{(n)}\subseteq A$. 
This holds for all non-zero $x\in D$, thus $A\subseteq \mathrm{in}_>(D)$. Here we necessarily have equality, as both spaces have the same dimension as $D$.   
\end{proof}


\section{Odd intersecting systems}\label{sec:intersecting}

Theorem~\ref{thm:monomialsubalg} or Proposition~\ref{prop:insubalg} open the way to reduce certain questions on square zero subspaces of $E_{\overline{1}}$ to questions about odd intersecting systems. 
Recall that a set $\mathcal{F}\subseteq 2^{[n]}$ of subsets of $[n]$ is called an {\it intersecting system} if $A\cap B\ne \emptyset$ for  any $A,B\in \mathcal{F}$, and 
it is an {\it odd intersecting sytem} if in addition we have that $|A|$ is odd for all $A\in\mathcal{F}$.

\begin{proposition}\label{prop:oddintersecting}
Let $\mathcal{F}\subseteq 2^{[n]}$ be an odd intersecting system.  
\begin{itemize}
\item [(i)] If $n$ is even then $|\mathcal{F}|\leq 2^{n-2}$. 
\item [(ii)] If $n$ is odd, $\mathcal{F}\subseteq \binom{[n]}{i}\cup\binom{[n]}{n-i-1}$ for some odd $i$ with $i<n/2-1$ and $\mathcal{F}$ is of maximal possible size,  
then $\mathcal{F}=\binom{[n]}{n-i-1}$. 
\item [(iii)] If $n=4k+1$ (where $k$ is a non-negative integer) and $|\mathcal{F}|$ is  maximal then $\mathcal{F}=\bigcup_{n/2<i \textrm{ odd}}\binom{[n]}{i}$.
\item [(iv)] If $n=4k+3$ (where $k$ is a non-negative integer) and $|\mathcal{F}|$ is maximal then $\mathcal{F}=\bigcup_{n/2<i \textrm{ odd}}\binom{[n]}{i}\cup \mathcal{G}$, where there exists an $l\in [n]$ such that $\mathcal{G}=\{X\in \binom{[n]}{2k+1}: l\in X\}$
\end{itemize}
\end{proposition}
\begin{proof}
(i) This follows from the fact that  if $n$ is even,  $X\subseteq [n]$ such that $|X|$ is odd,  then $|[n]\setminus X|$ is also odd, so $\mathcal{F}$ can not simultaneously contain 
$X$ and its complement. 

(ii)  Write $\mathcal{F}^c$ for the complement of $\mathcal{F}$ in $2^{[n]}$. 
We have the inclusion 
\begin{eqnarray*}
\{(A,B)\mid A\in\mathcal{F}\cap \binom{[n]}i, B\in \binom{[n]}{n-i-1}, A\cap B=\emptyset \} \\ 
\subseteq \{(A,B)\mid B\in \mathcal{F}^c\cap \binom{[n]}{n-i-1}, A\in\binom{[n]}i,  A\cap B=\emptyset\}.\end{eqnarray*}
It follows that 
\[|\mathcal{F}\cap \binom{[n]}i|(n-i)\le |\mathcal{F}^c\cap \binom{[n]}{n-i-1}|(i+1)\] 
and hence 
$|\mathcal{F}\cap \binom{[n]}i|\frac{n-i}{i+1}+|\mathcal{F}\cap \binom{[n]}{n-i-1}|\le \binom{n}{n-i-1}$. 
Since $\frac{n-i}{i+1}>1$, we get $|\mathcal{F}|\le \binom{n}{n-i-1}$ with equality only if 
$\mathcal{F}\subseteq \binom{[n]}{n-i-1}$. 
Note that since $n-i-1>n/2$, $\binom{n}{n-i-1}$ is an intersecting system.

(iii) follows from (ii).

(iv) follows from  (ii) and the  Erd\H os-Ko-Rado Theorem \cite{erdos-ko-rado}. 
\end{proof}


\section{Commutative subalgebras of maximal dimension}

\begin{theorem}\label{thm:maxcommsubalg}
Write $k$ for the lower integer part of $n/4$. 
\begin{itemize}
\item [(i)] The maximal dimension of a commutative subalgebra of $E^{(n)}$ is 
$\dim(E^{(n)}_{\overline{0}})+|\mathcal{F}|$ where $\mathcal{F}\subseteq 2^{[n]}$ is an odd intersecting system of maximal possible size, so this dimension is 
\[ \begin{array}{ll}
 3\cdot 2^{n-2}  & \mbox{ when }n\mbox{ is even}; \\
 2^{n-1}+\sum_{l=k}^{2k}\binom{n}{2l+1} & \mbox{ when }n=4k+1;\\
 2^{n-1}+\binom{n-1}{2k} +\sum_{l=k}^{2k}\binom{n}{2l+3}& \mbox{ when }n=4k+3. 
\end{array}
\]
\item [(ii)] When  $n$ is even, all maximal commutative subalgebras of $E^{(n)}$ have the same dimension, but they are not all isomorphic if $n>2$. 
\item[(iii)] When $n=4k+1$,   there is a unique maximal dimensional commutative subalgebra in $E^{(n)}$, namely  
$E_{\overline{0}}\oplus\bigoplus_{n/2<i\textrm{ odd}}E_i$. 
\item[(iv)] When $n=4k+3$, the maximal dimensional commutative subalgebras of $E^{(n)}$ are exactly the subspaces of the form 
$E_{\overline{0}}\oplus C\oplus\bigoplus_{n/2<i\textrm{ odd}}E_i$ where $C\subset E^{(n)}_{2k+1}$ is a square zero subspace of dimension $\binom{n-1}{2k}$. 
\item [(v)] When $n$ is odd, there exist maximal commutative subalgebras of $E^{(n)}$ that are not maximal dimensional commutative subalgebras in $E^{(n)}$. 
\end{itemize}
\end{theorem}

\begin{proof} 
(i) Note that a set of monomials $\{v_J\mid J\in\mathcal{F}\}$ spans a square zero subspace in $E_{\overline{1}}$ if and only if $\mathcal{F}\subseteq 2^{[n]}$ is an odd intersecting system. 
Now let $A$ be a maximal dimensional commutative subalgebra in $E$. Then in particular $A$ is a maximal commutative subalgebra of $E$, hence by Proposition~\ref{prop:strukt1} $A=E_{\overline{0}}\oplus D$ where $D\subseteq E_{\overline{1}}$ is a maximal dimensional square zero subspace. 
By Theorem~\ref{thm:monomialsubalg} $\gamma_1\dots\gamma_n(D)$ is a maximal dimensional square zero subspace of $E_{\overline{1}}$, hence 
$\mathcal{F}(D):=\{J\subseteq [n]\mid v_J\in \gamma_1\dots\gamma_n(D)\}$ is an odd intersecting system of maximal size. 
Since $|\mathcal{F}(D)|=\dim(D)$, statement (i) follows from Proposition~\ref{prop:oddintersecting}. 

(ii) It was shown in Corollary~\ref{cor:even} that for even $n$ the maximal commutative subalgebras all have the same dimension. 
Set $A:=E_{\overline{0}}+ \mathrm{Span}_{\mathbb{F}}\{v_J\mid 1\in J\}$ and 
$B:=E_{\overline{0}}\oplus (\oplus_{n/2<l\textrm{ odd }}E_l)\oplus C$ where $C=0$ when $n=4k$ and $C:=\mathrm{Span}_{\mathbb{F}}\{v_J\mid 1\in J,|J|=2k+1\}$ 
when $n=4k+2$. 
These algebras are local, their radical is their unique maximal ideal spanned by their homogeneous components of positive degree. We have 
$\mathrm{rad}(A)=\mathrm{rad}(A)^2\oplus E_2\oplus\mathbb{F}v_1$, showing $\dim(\mathrm{rad}(A)/\mathrm{rad}(A)^2)=\binom n2 +1$. 
On the other hand $\mathrm{rad}(B)=\mathrm{rad}(B)^2\oplus E_2\oplus E_{2k+1}$ when $n=4k$  whereas 
 $\mathrm{rad}(B)=\mathrm{rad}(B)^2\oplus E_2\oplus C\oplus \mathrm{Span}_{\mathbb{F}}\{v_J\mid |J|=2k+3,1\notin J\}$. 
 It follows that $\dim(\mathrm{rad}(B)/\mathrm{rad}(B)^2)=\binom n2+\binom{n}{2k+1}$ when $n=4k$ and  $\dim(\mathrm{rad}(B)/\mathrm{rad}(B)^2)=\binom n2+\binom{n-1}{2k}+\binom{n-1}{2k+3}$ when $n=4k+2$. This shows that $\dim(\mathrm{rad}(A)/\mathrm{rad}(A)^2)\ne \dim(\mathrm{rad}(B)/\mathrm{rad}(B)^2)$ for $n>2$, hence 
 $A\ncong B$ are nonisomorphic algebras. 

(iii) Assume $A\ne E_{\overline{0}}\oplus(\bigoplus_{n/2<i\textrm{ odd}}E_i)$ is another commutative subalgebra of maximal dimension. 
Since $A$ is maximal by Proposition~\ref{prop:strukt1} we get $A=E_{\overline{0}}\oplus D$, where $D\subseteq E_{\overline{1}}$ is a square zero subspace, which is not $\bigoplus_{n/2<i\text{ odd}}E_i$. So $D$ has an element  $x$ such that $x_J\ne 0$ for some $J\subseteq [n]$ with $|J|<n/2$. Suppose that $J$ was chosen here with $|J|$ minimal possible. 
Choose a permutation $\sigma\in S_n$ such that $\sigma(\{1,\dots,|J|\})=J$. Observe that $J\in \mathcal{F}_{\sigma}(D):=\{I\subseteq [n]\mid v_I\in 
\gamma_{\sigma(1)}\dots \gamma_{\sigma(n)}(D)\}$. 
As explained in (i) (and by Remark~\ref{remark:other-gamma}) 
$\mathcal{F}_{\sigma}(D)$ is a maximal size odd intersecting system. Thus $J\in \mathcal{F}_{\sigma}(D)$ contradicts Proposition~\ref{prop:oddintersecting} (iii). 

(iv) It is obvious that the subspaces in the statement are commutative subalgebras, and their dimension agrees with the value given in (i).  
A similar argument as in (iii) shows that by Proposition~\ref{prop:strukt1}, Theorem~\ref{thm:monomialsubalg}, Remark~\ref{remark:other-gamma}, and 
Proposition~\ref{prop:oddintersecting} a maximal dimensional commutative subalgebra $A$ is of the form $A=E_{\overline{0}}\oplus D$ where $D\subseteq E_{\overline{1}}$, $D^2=\{0\}$, and 
$D\subset \bigoplus_{i\ge 2k+1 \textrm{ odd}}E_i$. In addition, taking into account Propositions~\ref{prop:gamma-grading} and \ref{prop:a^min} we deduce $\dim(D^{\min}\cap E_{2k+1})=\binom{n-1}{2k}$. 
It follows that $\dim(D\cap \bigoplus_{i>2k+1\textrm{ odd}})=\dim(\bigoplus_{i>2k+1\textrm{ odd}} E_i)$. Consequently 
$D=C\oplus \bigoplus_{i>2k+1\textrm{ odd}} E_i$ where $C\subset E_{2k+1}$, and necessarily $C^2=\{0\}$, $\dim(C)=\binom{n-1}{2k}$.  

(v) Clearly $D:=\mathrm{Span}_{\mathrm{F}}\{v_J\mid 1\in J, |J| \textrm{ is odd}\}$ is a maximal square zero subspace in $E_{\overline{1}}$, hence 
$E_{\overline{0}}$ is a maximal commutative subalgebra. Its dimension is $3\cdot 2^{n-2}$, so when $n$ is odd this is strictly smaller than the maximal possible dimension of a commutative subalgebra of $E$, which is given in (i). 
\end{proof} 

\begin{conjecture}
If $n=4k+3$ and $A_1,A_2\subseteq E$ are maximal dimensional commutative subalgebras then $A_1\cong A_2$. 
\end{conjecture}

\begin{conjecture}\label{conj:2} 
If $n=4k+1$ and $A$ is a maximal commutative subalgebra of $E^{(n)}$ then dim$(A)\geq 3\cdot 2^{n-2}$.
\end{conjecture}

\begin{conjecture}\label{conj:3}
If $n=4k+3$ and $A$ is a maximal commutative subalgebra of $E^{(n)}$ then dim$(A)\geq 2^{n-1}+\sum_{l=k}^{2k}\binom{n}{2l+3}+|\mathcal{F}|$, such that $\mathcal{F}\subseteq \binom{[n]}{\lfloor n/2\rfloor}$ is a maximal  intersecting system of minimal possible size.
\end{conjecture}

Conjecture~\ref{conj:2} holds for $n=5$ and Conjecture~\ref{conj:3} holds for $n=7$ by Proposition~\ref{prop:contains-linear} below. 
We finish this section with a result that classifies maximal commutative subalgebras in $E$ of a special form: 

\begin{proposition}\label{prop:contains-linear}
The maximal commutative subalgebras of $E$ that contain an element whose degree $1$ component is non-zero are exactly the subalgebras of the form $\alpha(A)$, where 
$\alpha$ is an $\mathbb{F}$-algebra  automorphism of $E$, and $A=E_{\overline{0}}+\mathrm{Span}_{\mathbb{F}}\{v_J\mid 1\in J\subseteq [n]\}$ (in particular, these subalgebras have dimension $3\cdot 2^{n-2}$). 
\end{proposition} 

\begin{proof} Let $B$ be a maximal commutative subalgebra of $E$ containing an element $a$ whose degree $1$ component is non-zero. 
Then by Proposition~\ref{prop:strukt1} $A$ contains an element $b$ with $b\in E_{\overline{1}}$ and the degree $1$ component $b_1$ of $b$ is non-zero. 
A linear automorphism of $V$ sending $b_1$ to $v_1$ extends to an $\mathbb{F}$-algebra automorphism $\beta$ of $E$, and $\beta(B)$ contains 
$\beta(b)=v_1+c$ with $c\in E_3+E_5+E_7+\cdots$. It is also known (and easy to see) that the map $v_1+c\mapsto v_1$, $v_2\mapsto v_2$, $\dots$, $v_n\mapsto v_n$ extends to an $\mathbb{F}$-algebra automorphism $\rho$ of $E$ (see \cite{bavula} for more details on the automorphism group of $E$). Then $\rho\beta(B)$ is a maximal commutative subalgebra of $E$ containing $v_1$. If $D$ is square zero subspace of $E_{\overline{1}}$ containing $v_1$, then necessarily $D\subseteq v_1E\cap E_{\overline{1}}$. It follows by Proposition~\ref{prop:strukt1} that $\rho\beta(B)=A$, where $A$ is the subalgebra of $E$ in the statement.  
\end{proof} 


\section{Examples}

\begin{example}\label{ex:1} {\rm 
For $n=2$ and the subspace $W=\mathbb{F}(v_1+v_2)$ we have 
\[\gamma_1(\gamma_2(W))= \mathbb{F}v_1\ne\mathbb{F}v_2= \gamma_2(\gamma_1(W)).\]   }
\end{example}

\begin{example}\label{ex:2} {\rm 
Consider the  $6$-dimensional ideal $A:=\mathbb{F}(v_{\{1,2\}}+v_{\{3.4\}})\oplus E_3\oplus E_4$  of $E^{(4)}$. Then $\gamma_1(A)=\mathbb{F}v_{\{3,4\}}\oplus E_3\oplus E_4$. Since $A^2=\mathbb{F}v_{\{1,2,3,4\}}$ and $\gamma_1(A)^2=0$ we get $A\ncong \gamma_1(A)$ and $\gamma_1(A)^2\subsetneq\gamma_1(A^2)$. }
\end{example}

\begin{example}\label{ex:3}{\rm 
There is a square zero subspace $D:=\mathrm{Span}_{\mathbb{F}}\{v_{\{1,2,3\}}+v_{\{4,5,6\}},v_{\{1,2,4\}}+v_{\{3,5,6\}}\}$ of $E^{(6)}$ and permutations $\sigma,\rho\in S_6$ such that 
$\mathcal{F}_{\sigma}(D)$ and $\mathcal{F}_{\rho}(D)$ (see the proof of Theorem~\ref{thm:maxcommsubalg} (iii) for this notation) 
are not isomorphic as intersecting systems. Indeed, for 
$\sigma=\mathrm{id}$, $\rho=(36)$ (transposition) we have 
\[\mathcal{F}_{\sigma}(D)=\{\{1,2,3\},\{1,2,4\}\} \mbox{ and } 
\mathcal{F}_{\rho}(D)=\{\{4,5,6\},\{1,2,4\}\}.\] 
The two sets in $\mathcal{F}_{\sigma}(D)$ have two common elements, whereas the two sets in $\mathcal{F}_{\rho}(D)$ have one common element. }
\end{example}

\begin{example}\label{ex:4}{\rm 
For the $n$-dimensional subspace 
$D:=\mathrm{Span}_{\mathbb{F}}\{\sum_{I\subseteq \binom{[n]}{k}} v_I \mid k\in [n]\}\subset E^{(n)}$ and a permutation $\sigma\in S_n$ 
 we have 
 \[\gamma_{\sigma(1)}\gamma_{\sigma(2)}\dots\gamma_{\sigma(n)}(D)=\mathrm{Span}_{\mathbb{F}}\{ v_{\{\sigma(1)\}},v_{\{\sigma(1),\sigma(2)\}},v_{\{\sigma(1),\sigma(2),\sigma(3)\}},\dots, v_{[n]}\}.\]
So if $\sigma\ne\rho\in S_n$ then
$\gamma_{\sigma(1)}\gamma_{\sigma(2)}\dots\gamma_{\sigma(n)}(D)\ne\gamma_{\rho(1)}\gamma_{\rho(2)}\dots\gamma_{\rho(n)}(D).$}
\end{example}

\begin{example}\label{ex:6} {\rm 
Consider the $4$-dimensional subalgebra  $D:=\mathrm{Span}_{\mathbb{F}}\{v_{\{1,2\}}+v_3,v_1,v_{\{1,3\}},v_{\{1,2,3\}}\}$  of $E^{(3)}$. 
Then $D$, $\gamma_1\gamma_2\gamma_3(D)$ and $\gamma_3\gamma_2\gamma_1(D)$ are pairwise non-isomorphic  subalgebras, since the dimensions of their squares are $2$, $0$, and $1$. }
\end{example}

\begin{example}\label{example:nonmonomial} {\rm For $x=\sum_{J\in\binom{[n]}{2}}x_Jv_J\in E^{(n)}_2$ we have $x^2=0$ if and only if 
\[x_{\{i,j\}}x_{\{k,l\}}-x_{\{i,k\}}x_{\{j,l\}}+x_{\{i,l\}}x_{\{j,k\}}=0 \mbox{ holds for all }\{i,j,k,l\}\in\binom{[n]}{4}.\] 
These are the well-known {\it Grassmann-Pl\"ucker relations} (see for example \cite{sturmfels}), so $x^2=0$ if and only if $x\in E^{(n)}_2\cong \bigwedge^2V$ is {\it decomposable}, i.e. $x\in D:=\{yz\mid y,z\in E^{(n)}_1\}$. 
The subset $D$ is Zariski closed in the $\binom{n}{2}$-dimensional affine space $E^{(n)}_2$, and the dimension of this affine algebraic variety is $2n-3$. 
It follows that when $n\ge 4$, for a general linear subspace $L\subset E^{(n)}_2$ with $\dim(L)\le \binom{n}{2}-2n+3$ we have $L\cap D=\{0\}$. 
Now let $A$ be the subalgebra of $E$ generated by such a non-zero $L$. Then $A\setminus A^2=(L\setminus\{0\})+A^2$ contains no non-zero element $a$ with $a^2=0$, so 
$A$ can not be isomorphic to a subalgebra $B$ of $E$ generated by monomials, since a minimal set of monomials generating $B$ consists of square zero elements in $B\setminus B^2$. For example, in $E^{(4)}$ take $L:=\mathbb{F}(v_{\{1,2\}}+v_{\{3,4\}})$; the subalgebra generated by $L$ is  $\mathrm{Span}_{\mathbb{F}}\{v_{\{1,2\}}+v_{\{3,4\}},v_{\{1,2,3,4\}}\}$. The square zero elements span a proper subspace in it, so it is not isomorphic to a subalgebra of $E^{(n)}$ generated by monomials. }
\end{example}  

\begin{example}\label{example:nongraded} {\rm Set $A:=\mathrm{Span}_{\mathbb{F}}\{v_1+v_{\{2,3\}},v_{\{1,2,3\}}\}$. We have $(v_1+v_{\{2,3\}})^2=2v_{\{1,2,3\}}$, 
$v_{\{1,2,3\}}^2=0$, $(v_1+v_{\{2,3\}})v_{\{1,2,3\}}=0$. So $A$ is a $2$-dimensional nilpotent subalgebra of $E^{(3)}$, containing an element whose square is not zero. 
It is easy to see that a $2$-dimensional nilpotent graded subalgebra of  $E^{(3)}$ must be a square zero subspace, so $A$ is not isomorphic to a graded subalgebra of $E^{(3)}$. On the other hand, $A$ is isomorphic as an $\mathbb{F}$-algebra to the subalgebra $\mathrm{Span}_{\mathbb{F}}\{v_{\{1,2\}}+v_{\{3,4\}},v_{\{1,2,3,4\}}\}$ of $E^{(4)}$. }
\end{example} 


\begin{center} Acknowledgement \end{center} 

We are grateful to J. Szigeti for an inspiring lecture and some discussions.



\begin{thebibliography}{mmm} 

\bibitem{aramova-herzog-hibi} A. Aramova, J. Herzog and T. Hibi, Gotzmann theorems for exterior algebras and combinatorics, J. Algebra 191 (1997), 174-211. 

\bibitem{bavula} V.  V. Bavula, 
The Jacobian map, the Jacobian group and the group of automorphisms of the Grassmann algebra,  
Bull. Soc. Math. France 138 (2010), no. 1, 39-117. 

\bibitem{erdos-ko-rado} P. Erd\H os, C. Ko, R. Rado, Intersection theorems for systems of finite
sets, Quart. J. Math. Oxford, ser. (2) 12 (1961), 313-318.

\bibitem{gustafson} W. H. Gustafson, On maximal commutative algebras of linear transformations, J. Algebra 42 (1976), 557-563. 

\bibitem{jacobson} N. Jacobson, Schur's theorems on commutative matrices, Bull. Amer. Math. Soc. 50 (1944), 431-436. 

\bibitem{szigeti-etal} L. M\'arki, J. Meyer, J. Szigeti, L. van Wyk, Matrix representations of finitely generated Grassmann algebras and some consequences, arXiv:1307.0292, 
Israel J. Math., to appear. 

\bibitem{schur} I. Schur, Zur Theorie der vertauschbaren Matrizen, J. Reine Angew. Math. 130 (1905), 66-76. 

\bibitem{sturmfels} B. Sturmfels, Algorithms in Invariant Theory, 
Springer-Verlag, Wien, 1993. 

\end{thebibliography}
\end{document}